\documentclass{article}
\usepackage{graphicx,amsmath,amsthm,amsopn,amstext,amscd,amsfonts,amssymb}
\usepackage{natbib}

\def \u {\mathop{\rm \mathcal{U}}\nolimits}
\def \tr {\mathop{\rm tr}\nolimits}
\def \re {\mathop{\rm Re}\nolimits}

\def \Vol {\mathop{\rm Vol}\nolimits}

\def \etr {\mathop{\rm etr}\nolimits}
\def \diag {\mathop{\rm diag}\nolimits}

\renewenvironment{abstract}
                 {\vspace{6pt}
                  \begin{center}
                  \begin{minipage}{5in}
                  \centerline{\textbf{Abstract}}
                  \noindent\ignorespaces
                 }
                 {\end{minipage}\end{center}}
\newtheorem{theorem}{\textbf{Theorem}}[section]
\newtheorem{corollary}{\textbf{Corollary}}[section]

\newtheorem{proposition}{\textbf{Proposition}}[section]
\theoremstyle{definition}
\newtheorem{definition}{\textbf{Definition}}[section]

\newtheorem{remark}{\textbf{Remark}}[section]

\title{\Large \textbf{Generalised matricvariate Pearson type II- distribution}}
\author{
  \textbf{Jos\'e A. D\'{\i}az-Garc\'{\i}a} \thanks{Corresponding author\newline
    }\\
  {\normalsize Universidad Aut\'onoma Agraria Antonio Narro}\\
  {\normalsize Calzada Antonio Narro 1923}\\
  {\normalsize 25350 Buenavista, Saltillo, Coahuila, M\'exico} \\
  {\normalsize E-mail: jadiaz@uaaan.mx} \\[2ex]
  \textbf{Francisco J. Caro-Lopera} \\
  {\normalsize Department of Basic Sciences} \\
  {\normalsize Universidad de Medell\'{\i}n} \\
  {\normalsize Carrera 87 No.30-65, of. 5-103}\\
  {\normalsize Medell\'{\i}n, Colombia}\\
  {\normalsize E-mail: fjcaro@udem.edu.co}\\
}
\date{}
\begin{document}
\maketitle

\begin{abstract}
The so called Pearson Type II-Riesz distribution, based on the Kotz-Riesz distribution is derived in this
paper. Specifically, the central nonsingular matricvariate  generalised Pearson type II-Riesz
distribution and beta-Riesz type I distributions for real normed division algebras are obtained.
\end{abstract}

\section{Introduction}\label{sec1}
Matrix distribution theory involves a number of interesting and applied approaches, from old to new
settings, integrated and unified techniques have appear recently in literature.

For example,  recall that if $\mathbf{X}$ and $\mathbf{U}_{1}$ are random matrices independently
distributed as matrix multivariate normal distribution and a Whishart distribution, respectively; it is
known that the random matrix $\mathbf{R} = \mathbf{L}^{-1}\mathbf{X},$ where $\mathbf{L}$ is any square
root of $\mathbf{U}=\mathbf{L}^{*}\mathbf{L} = \mathbf{U}_{1} + \mathbf{X}^{*}\mathbf{X}$, has a
matricvariate Pearson type II distribution. In the real case under normality, the \emph{matricvariate
Pearson type II distribution} (also known in the literature as \emph{matricvariate inverted $T$
distribution}) was studied in detail by \cite{di:67}, see also \cite{p:82}. This distribution was
previously studied by \cite{k:59}, also in the real case. For \emph{real, complex, quaternion and
octonion} cases this distribution was studied by \cite{dggj:12}.

In Bayesian inference, the matricvariate Pearson type II distribution is assumed as the sampling
distribution; then, considering a noninformative prior distribution, the posterior distribution and
marginal distributions, the posterior mean and generalised maximum likelihood estimators of the
parameters involved are found, \cite{fl:99}. The matricvariate Pearson type II distribution appears in
the frequentist approach to normal regression as the distribution of the Studentised error, see
\cite{jdggj:06} and \cite{kn:04}. Another opportunities and potential studies concerns the role of the
matricvariate Pearson type II distribution in multivariate analysis, because if the matrix $\mathbf{R}$
has a matricvariate Pearson type II distribution, then the matrix $\mathbf{R}^{*}\mathbf{R}$  is
distributed as matrix multivariate beta type I; and the distribution of the latter, in particular, plays
a fundamental role in the MANOVA model, see \cite{k:59,k:70} and \cite{m:82}.

A family of distributions on symmetric cones, termed the \textit{matrix multivariate Riesz
distributions}, was first introduced by \citet{hl:01} under the name of Riesz natural exponential family
(Riesz NEF); it was based on a special case of the so-termed Riesz measure from \citet[p. 137]{fk:94},
going back to \citet{r:49}. This Riesz distribution generalises the matrix multivariate gamma and Wishart
distributions, containing them as particular cases. Subsequently, \citet{dg:15c} and \citet{dg:15a}
proposes two versions of the Riesz distribution and two generalised of a class of type Kotz
distributions. This last two generalised type Kotz distributions are termed \emph{matrix multivariate
Kotz-Riesz distribution} and generalise the matrix multivariate normal distribution, containing this as
particular case.

In this point, we are able to propose a generalisation of the matrcvariate Pearson type II distribution.
With this as aim, let $\mathbf{R} = \mathbf{X}\mathbf{L}^{-1}$, where $\mathbf{L}$ is a upper triangular
matrix such that $\mathbf{U}=\mathbf{L}^{*}\mathbf{L} = \mathbf{U}_{1} + \mathbf{X}^{*}\mathbf{X}$, where
now we shall assume that $\mathbf{X}$ and $\mathbf{U}_{1}$ are random matrices independently distributed
as a matrix multivariate Kotz-Riesz distribution and a matrix multivariate Riesz distribution,
respectively. Then, the distribution of $\mathbf{R}$ shall be termed \emph{matricvariate Pearson type
II-Riesz distribution}.

Although during the 90's and 2000's were obtained important results in theory of random matrices
distributions, the  past 30 years have reached a substantial development. Essentially, these advances
have been archived through two approaches based on the \emph{theory of Jordan algebras} and the \emph{
theory of real normed division algebras}. A basic source of the mathematical tools of theory of random
matrices distributions under Jordan algebras can be found in \citet{fk:94}; and specifically, some works
in the context of theory of random matrices distributions based on Jordan algebras are provided in
\citet{m:94}, \citet{cl:96}, \citet{hl:01}, and \citet{hlz:05}, and the references therein. Parallel
results on theory of random matrices distributions based on real normed division algebras have been also
developed in random matrix theory and statistics, see \citet{gr:87}, \citet{d:02}, \citet {f:05},
\citet{dggj:11}, \citet{dggj:13}, among others. In  addition, from mathematical point of view, several
basic properties of the matrix multivariate Riesz distribution under \emph{the structure theory of normal
$j$-algebras}  and under \emph{theory of Vinberg algebras} in place of Jordan algebras have been studied,
see \citet{i:00} and \citet{bh:09}, respectively.

From a applied point of view, the relevance of \emph{the octonions} remains unclear. An excellent review
of the history, construction and many other properties of octonions is given in \citet{b:02}, where it is
stated that:
\begin{center}
\begin{minipage}[t]{4in}
\begin{sl}
``Their relevance to geometry was quite obscure until 1925, when \'Elie Cartan described `triality' --
the symmetry between vector and spinors in 8-dimensional Euclidian space. Their potential relevance to
physics was noticed in a 1934 paper by Jordan, von Neumann and Wigner on the foundations of quantum
mechanics...Work along these lines continued quite slowly until the 1980s, when it was realised that the
octionions explain some curious features of string theory... \textbf{However, there is still no
\emph{proof} that the octonions are useful for understanding the real world}. We can only hope that
eventually this question will be settled one way or another."
\end{sl}
\end{minipage}
\end{center}

For the sake of completeness, in the present article the case of octonions is considered, but the
veracity of the results obtained for this case can only be conjectured. Nonetheless, \citet[Section
1.4.5, pp. 22-24]{f:05} it is proved that the bi-dimensional density function of the eigenvalue, for a
Gaussian ensemble of a $2 \times 2$ octonionic matrix, is obtained from the general joint density
function of the eigenvalues for the Gaussian ensemble, assuming $m = 2$ and $\beta = 8$, see Section
\ref{sec2}. Moreover, as is established in \citet{fk:94} and \citet{S:97} the result obtained in this
article are valid for the \emph{algebra of Albert}, that is when hermitian matrices ($\mathbf{S}$) or
hermitian product of matrices ($\mathbf{X}^{*}\mathbf{X}$) are $3 \times 3$ octonionic matrices.

The present article is organised as follows; basic concepts and the notation of abstract algebra and
Jacobians are summarised in Section \ref{sec2}. Then Section \ref{sec3}  derives the nonsingular central
matricvariate Pearson type II-Riesz and the beta type I distributions, and some basic properties are also
studied. It should be noted that all these results are derived in the context of  real normed division
algebras, a useful integrated and unified approach recently implemented in matrix distribution theory.

\section{Preliminary results}\label{sec2}

A detailed discussion of real normed division algebras may be found in \cite{b:02} and \cite{E:90}. For
convenience, we shall introduce some notation, although in general we adhere to standard notation forms.

For our purposes: Let $\mathbb{F}$ be a field. An \emph{algebra} $\mathfrak{A}$ over $\mathbb{F}$ is a
pair $(\mathfrak{A};m)$, where $\mathfrak{A}$ is a \emph{finite-dimensional vector space} over
$\mathbb{F}$ and \emph{multiplication} $m : \mathfrak{A} \times \mathfrak{A} \rightarrow A$ is an
$\mathbb{F}$-bilinear map; that is, for all $\lambda \in \mathbb{F},$ $x, y, z \in \mathfrak{A}$,
\begin{eqnarray*}
  m(x, \lambda y + z) &=& \lambda m(x; y) + m(x; z) \\
  m(\lambda x + y; z) &=& \lambda m(x; z) + m(y; z).
\end{eqnarray*}
Two algebras $(\mathfrak{A};m)$ and $(\mathfrak{E}; n)$ over $\mathbb{F}$ are said to be
\emph{isomorphic} if there is an invertible map $\phi: \mathfrak{A} \rightarrow \mathfrak{E}$ such that
for all $x, y \in \mathfrak{A}$,
$$
  \phi(m(x, y)) = n(\phi(x), \phi(y)).
$$
By simplicity, we write $m(x; y) = xy$ for all $x, y \in \mathfrak{A}$.

Let $\mathfrak{A}$ be an algebra over $\mathbb{F}$. Then $\mathfrak{A}$ is said to be
\begin{enumerate}
  \item \emph{alternative} if $x(xy) = (xx)y$ and $x(yy) = (xy)y$ for all $x, y \in \mathfrak{A}$,
  \item \emph{associative} if $x(yz) = (xy)z$ for all $x, y, z \in \mathfrak{A}$,
  \item \emph{commutative} if $xy = yx$ for all $x, y \in \mathfrak{A}$, and
  \item \emph{unital} if there is a $1 \in \mathfrak{A}$ such that $x1 = x = 1x$ for all $x \in \mathfrak{A}$.
\end{enumerate}
If $\mathfrak{A}$ is unital, then the identity 1 is uniquely determined.

An algebra $\mathfrak{A}$ over $\mathbb{F}$ is said to be a \emph{division algebra} if $\mathfrak{A}$ is
nonzero and $xy = 0_{\mathfrak{A}} \Rightarrow x = 0_{\mathfrak{A}}$ or $y = 0_{\mathfrak{A}}$ for all
$x, y \in \mathfrak{A}$.

The term ``division algebra", comes from the following proposition, which shows that, in such an algebra,
left and right division can be unambiguously performed.

Let $\mathfrak{A}$ be an algebra over $\mathbb{F}$. Then $\mathfrak{A}$ is a division algebra if, and
only if, $\mathfrak{A}$ is nonzero and for all $a, b \in \mathfrak{A}$, with $b \neq 0_{\mathfrak{A}}$,
the equations $bx = a$ and $yb = a$ have unique solutions $x, y \in \mathfrak{A}$.

In the sequel we assume $\mathbb{F} = \Re$ and consider classes of division algebras over $\Re$ or
``\emph{real division algebras}" for short.

We introduce the algebras of \emph{real numbers} $\Re$, \emph{complex numbers} $\mathfrak{C}$,
\emph{quaternions} $\mathfrak{H}$ and \emph{octonions} $\mathfrak{O}$. Then, if $\mathfrak{A}$ is an
alternative real division algebra, then $\mathfrak{A}$ is isomorphic to $\Re$, $\mathfrak{C}$,
$\mathfrak{H}$ or $\mathfrak{O}$.

Let $\mathfrak{A}$ be a real division algebra with identity $1$. Then $\mathfrak{A}$ is said to be
\emph{normed} if there is an inner product $(\cdot, \cdot)$ on $\mathfrak{A}$ such that
$$
  (xy, xy) = (x, x)(y, y) \qquad \mbox{for all } x, y \in \mathfrak{A}.
$$
If $\mathfrak{A}$ is a \emph{real normed division algebra}, then $\mathfrak{A}$ is isomorphic $\Re$,
$\mathfrak{C}$, $\mathfrak{H}$ or $\mathfrak{O}$.

There are exactly four normed division algebras: real numbers ($\Re$), complex numbers ($\mathfrak{C}$),
quaternions ($\mathfrak{H}$) and octonions ($\mathfrak{O}$), see \cite{b:02}. We take into account that
should be taken into account, $\Re$, $\mathfrak{C}$, $\mathfrak{H}$ and $\mathfrak{O}$ are the only
normed division algebras; furthermore, they are the only alternative division algebras.

Let $\mathfrak{A}$ be a division algebra over the real numbers. Then $\mathfrak{A}$ has dimension either
1, 2, 4 or 8. In other branches of mathematics, the parameters $\alpha = 2/\beta$ and $t = \beta/4$ are
used, see \cite{er:05} and \cite{k:84}, respectively.

Finally, observe that

\begin{tabular}{c}
  $\Re$ is a real commutative associative normed division algebras, \\
  $\mathfrak{C}$ is a commutative associative normed division algebras,\\
  $\mathfrak{H}$ is an associative normed division algebras, \\
  $\mathfrak{O}$ is an alternative normed division algebras. \\
\end{tabular}

Let ${\mathcal L}^{\beta}_{n,m}$ be the set of all $n \times m$ matrices of rank $m \leq n$ over
$\mathfrak{A}$ with $m$ distinct positive singular values, where $\mathfrak{A}$ denotes a \emph{real
finite-dimensional normed division algebra}. Let $\mathfrak{A}^{n \times m}$ be the set of all $n \times
m$ matrices over $\mathfrak{A}$. The dimension of $\mathfrak{A}^{n \times m}$ over $\Re$ is $\beta mn$.
Let $\mathbf{A} \in \mathfrak{A}^{n \times m}$, then $\mathbf{A}^{*} = \bar{\mathbf{A}}^{T}$ denotes the
usual conjugate transpose.

Table \ref{table1} sets out the equivalence between the same concepts in the four normed division
algebras.

\begin{table}[th]
  \centering
  \caption{\scriptsize Notation}\label{table1}
  \begin{scriptsize}
  \begin{tabular}{cccc|c}
    \hline
    Real & Complex & Quaternion & Octonion & \begin{tabular}{c}
                                               Generic \\
                                               notation \\
                                             \end{tabular}\\
    \hline
    Semi-orthogonal & Semi-unitary & Semi-symplectic & \begin{tabular}{c}
                                                         Semi-exceptional \\
                                                         type \\
                                                       \end{tabular}
      & $\mathcal{V}_{m,n}^{\beta}$ \\
    Orthogonal & Unitary & Symplectic & \begin{tabular}{c}
                                                         Exceptional \\
                                                         type \\
                                                       \end{tabular} & $\mathfrak{U}^{\beta}(m)$ \\
    Symmetric & Hermitian & \begin{tabular}{c}
                              Quaternion \\
                              hermitian \\
                            \end{tabular}
     & \begin{tabular}{c}
                              Octonion \\
                              hermitian \\
                            \end{tabular} & $\mathfrak{S}_{m}^{\beta}$ \\
    \hline
  \end{tabular}
  \end{scriptsize}
\end{table}

We denote by ${\mathfrak S}_{m}^{\beta}$ the real vector space of all $\mathbf{S} \in \mathfrak{A}^{m
\times m}$ such that $\mathbf{S} = \mathbf{S}^{*}$. In addition, let $\mathfrak{P}_{m}^{\beta}$ be the
\emph{cone of positive definite matrices} $\mathbf{S} \in \mathfrak{A}^{m \times m}$. Thus,
$\mathfrak{P}_{m}^{\beta}$ consist of all matrices $\mathbf{S} = \mathbf{X}^{*}\mathbf{X}$, with
$\mathbf{X} \in \mathfrak{L}^{\beta}_{n,m}$; then $\mathfrak{P}_{m}^{\beta}$ is an open subset of
${\mathfrak S}_{m}^{\beta}$.

Let $\mathfrak{D}_{m}^{\beta}$ consisting of all $\mathbf{D} \in \mathfrak{A}^{m \times m}$, $\mathbf{D}
= \diag(d_{1}, \dots,d_{m})$. Let $\mathfrak{T}_{U}^{\beta}(m)$ be the subgroup of all \emph{upper
triangular} matrices $\mathbf{T} \in \mathfrak{A}^{m \times m}$ such that $t_{ij} = 0$ for $1 < i < j
\leq m$.

For any matrix $\mathbf{X} \in \mathfrak{A}^{n \times m}$, $d\mathbf{X}$ denotes the\emph{ matrix of
differentials} $(dx_{ij})$. Finally, we define the \emph{measure} or volume element $(d\mathbf{X})$ when
$\mathbf{X} \in \mathfrak{A}^{n \times m}, \mathfrak{S}_{m}^{\beta}$, $\mathfrak{D}_{m}^{\beta}$ or
$\mathcal{V}_{m,n}^{\beta}$, see \cite{dggj:11} and \cite{dggj:13}.

If $\mathbf{X} \in \mathfrak{A}^{n \times m}$ then $(d\mathbf{X})$ (the Lebesgue measure in
$\mathfrak{A}^{n \times m}$) denotes the exterior product of the $\beta mn$ functionally independent
variables
$$
  (d\mathbf{X}) = \bigwedge_{i = 1}^{n}\bigwedge_{j = 1}^{m}dx_{ij} \quad \mbox{ where }
    \quad dx_{ij} = \bigwedge_{k = 1}^{\beta}dx_{ij}^{(k)}.
$$

If $\mathbf{S} \in \mathfrak{S}_{m}^{\beta}$ (or $\mathbf{S} \in \mathfrak{T}_{U}^{\beta}(m)$ with
$t_{ii} >0$, $i = 1, \dots,m$) then $(d\mathbf{S})$ (the Lebesgue measure in $\mathfrak{S}_{m}^{\beta}$
or in $\mathfrak{T}_{U}^{\beta}(m)$) denotes the exterior product of the exterior product of the
$m(m-1)\beta/2 + m$ functionally independent variables,
$$
  (d\mathbf{S}) = \bigwedge_{i=1}^{m} ds_{ii}\bigwedge_{i > j}^{m}\bigwedge_{k = 1}^{\beta}
                      ds_{ij}^{(k)}.
$$
Observe, that for the Lebesgue measure $(d\mathbf{S})$ defined thus, it is required that $\mathbf{S} \in
\mathfrak{P}_{m}^{\beta}$, that is, $\mathbf{S}$ must be a non singular Hermitian matrix (Hermitian
definite positive matrix).

If $\mathbf{\Lambda} \in \mathfrak{D}_{m}^{\beta}$ then $(d\mathbf{\Lambda})$ (the Legesgue measure in
$\mathfrak{D}_{m}^{\beta}$) denotes the exterior product of the $\beta m$ functionally independent
variables
$$
  (d\mathbf{\Lambda}) = \bigwedge_{i = 1}^{n}\bigwedge_{k = 1}^{\beta}d\lambda_{i}^{(k)}.
$$
If $\mathbf{H}_{1} \in \mathcal{V}_{m,n}^{\beta}$ then
$$
  (\mathbf{H}^{*}_{1}d\mathbf{H}_{1}) = \bigwedge_{i=1}^{m} \bigwedge_{j =i+1}^{n}
  \mathbf{h}_{j}^{*}d\mathbf{h}_{i}.
$$
where $\mathbf{H} = (\mathbf{H}^{*}_{1}|\mathbf{H}^{*}_{2})^{*} = (\mathbf{h}_{1}, \dots,
\mathbf{h}_{m}|\mathbf{h}_{m+1}, \dots, \mathbf{h}_{n})^{*} \in \mathfrak{U}^{\beta}(n)$. It can be
proved that this differential form does not depend on the choice of the $\mathbf{H}_{2}$ matrix. When $n
= 1$; $\mathcal{V}^{\beta}_{m,1}$ defines the unit sphere in $\mathfrak{A}^{m}$. This is, of course, an
$(m-1)\beta$- dimensional surface in $\mathfrak{A}^{m}$. When $n = m$ and denoting $\mathbf{H}_{1}$ by
$\mathbf{H}$, $(\mathbf{H}d\mathbf{H}^{*})$ is termed the \emph{Haar measure} on
$\mathfrak{U}^{\beta}(m)$.

The surface area or volume of the Stiefel manifold $\mathcal{V}^{\beta}_{m,n}$ is
\begin{equation}\label{vol}
    \Vol(\mathcal{V}^{\beta}_{m,n}) = \int_{\mathbf{H}_{1} \in
  \mathcal{V}^{\beta}_{m,n}} (\mathbf{H}_{1}d\mathbf{H}^{*}_{1}) =
  \frac{2^{m}\pi^{mn\beta/2}}{\Gamma^{\beta}_{m}[n\beta/2]},
\end{equation}
where $\Gamma^{\beta}_{m}[a]$ denotes the multivariate \emph{Gamma function} for the space
$\mathfrak{S}_{m}^{\beta}$. This can be obtained as a particular case of the \emph{generalised gamma
function of weight $\kappa$} for the space $\mathfrak{S}^{\beta}_{m}$ with $\kappa = (k_{1}, k_{2},
\dots, k_{m}) \in \Re^{m}$, taking $\kappa =(0,0,\dots,0) \in \Re^{m}$ and which for $\re(a) \geq
(m-1)\beta/2 - k_{m}$ is defined by, see \cite{gr:87} and \citet{fk:94},
\begin{eqnarray}\label{int1}
  \Gamma_{m}^{\beta}[a,\kappa] &=& \displaystyle\int_{\mathbf{A} \in \mathfrak{P}_{m}^{\beta}}
  \etr\{-\mathbf{A}\} |\mathbf{A}|^{a-(m-1)\beta/2 - 1} q_{\kappa}(\mathbf{A}) (d\mathbf{A}) \\
&=& \pi^{m(m-1)\beta/4}\displaystyle\prod_{i=1}^{m} \Gamma[a + k_{i}
    -(i-1)\beta/2]\nonumber\\ \label{gammagen1}
&=& [a]_{\kappa}^{\beta} \Gamma_{m}^{\beta}[a],
\end{eqnarray}
where $\etr(\cdot) = \exp(\tr(\cdot))$, $|\cdot|$ denotes the determinant, and for $\mathbf{A} \in
\mathfrak{S}_{m}^{\beta}$
\begin{equation}\label{hwv}
    q_{\kappa}(\mathbf{A}) = |\mathbf{A}_{m}|^{k_{m}}\prod_{i = 1}^{m-1}|\mathbf{A}_{i}|^{k_{i}-k_{i+1}}
\end{equation}
with $\mathbf{A}_{p} = (a_{rs})$, $r,s = 1, 2, \dots, p$, $p = 1,2, \dots, m$ is termed the \emph{highest
weight vector}, see \cite{gr:87}. Also,
\begin{eqnarray*}
  \Gamma_{m}^{\beta}[a] &=& \displaystyle\int_{\mathbf{A} \in \mathfrak{P}_{m}^{\beta}}
  \etr\{-\mathbf{A}\} |\mathbf{A}|^{a-(m-1)\beta/2 - 1}(d\mathbf{A}) \\ \label{cgamma}
    &=& \pi^{m(m-1)\beta/4}\displaystyle\prod_{i=1}^{m} \Gamma[a-(i-1)\beta/2],
\end{eqnarray*}
and $\re(a)> (m-1)\beta/2$.

In other branches of mathematics the \textit{highest weight vector} $q_{\kappa}(\mathbf{A})$ is also
termed the \emph{generalised power} of $\mathbf{A}$ and is denoted as $\Delta_{\kappa}(\mathbf{A})$, see
\cite{fk:94} and \cite{hl:01}.

Additional properties of $q_{\kappa}(\mathbf{A})$, which are immediate consequences of the definition of
$q_{\kappa}(\mathbf{A})$ are:
\begin{enumerate}
  \item Let $\mathbf{A} = \mathbf{L}^{*}\mathbf{DL}$ be the L'DL decomposition of $\mathbf{A} \in \mathfrak{P}_{m}^{\beta}$,
        where $\mathbf{L} \in \mathfrak{T}_{U}^{\beta}(m)$ with $l_{ii} = 1$, $i = 1, 2, \ldots ,m$ and
        $\mathbf{D} = \diag(\lambda_{1}, \dots, \lambda_{m})$, $\lambda_{i} \geq 0$, $i = 1, 2, \ldots
        ,m$. Then
        \begin{equation}\label{qk1}
          q_{\kappa}(\mathbf{A}) = \prod_{i=1}^{m} \lambda_{i}^{k_{i}}.
        \end{equation}
      \item
      \begin{equation}\label{qk2}
        q_{\kappa}(\mathbf{A}^{-1}) =  q_{-\kappa^{*}}^{*}(\mathbf{A}),
      \end{equation}
      where $\kappa^{*}=(k_{m}, k_{m-1}, \dots,k_{1})$, $-\kappa^{*}=(-k_{m}, -k_{m-1},
      \dots,-k_{1})$,
      \begin{equation}\label{hhwv}
         q_{\kappa}^{*}(\mathbf{A}) = |\mathbf{A}_{m}|^{k_{m}}\prod_{i = 1}^{m-1}|\mathbf{A}_{i}|^{k_{i}-k_{i+1}}
      \end{equation}
      and
      \begin{equation}\label{qqk1}
        q_{\kappa}^{*}(\mathbf{A}) = \prod_{i=1}^{m} \lambda_{i}^{k_{m-i+1}},
      \end{equation}
      see \citet[pp. 126-127 and Proposition VII.1.5]{fk:94}.

  Alternatively, let $\mathbf{A} = \mathbf{T}^{*}\mathbf{T}$ the Cholesky's decomposition of
  matrix $\mathbf{A} \in \mathfrak{P}_{m}^{\beta}$, with $\mathbf{T}=(t_{ij}) \in
  \mathfrak{T}_{U}^{\beta}(m)$, then $\lambda_{i} = t_{ii}^{2}$, $t_{ii} \geq 0$, $i = 1, 2,
  \ldots ,m$. See \citet[p. 931, first paragraph]{hl:01}, \citet[p. 390, lines -11 to
  -16]{hlz:05} and \citet[p.5, lines 1-6]{k:14}.
  \item if $\kappa = (p, \dots, p)$, then
    \begin{equation}\label{qk3}
        q_{\kappa}(\mathbf{A}) = |\mathbf{A}|^{p},
    \end{equation}
    in particular if $p=0$, then $q_{\kappa}(\mathbf{A}) = 1$.
  \item if $\tau = (t_{1}, t_{2}, \dots, t_{m})$, $t_{1}\geq t_{2}\geq \cdots \geq t_{m} \geq
  0$, then
    \begin{equation}\label{qk41}
        q_{\kappa+\tau}(\mathbf{A}) = q_{\kappa}(\mathbf{A})q_{\tau}(\mathbf{A}),
    \end{equation}
    in particular if $\tau = (p,p, \dots, p)$,  then
    \begin{equation}\label{qk42}
        q_{\kappa+\tau}(\mathbf{A}) \equiv q_{\kappa+p}(\mathbf{A}) = |\mathbf{A}|^{p} q_{\kappa}(\mathbf{A}).
    \end{equation}
    \item Finally, for $\mathbf{B} \in \mathfrak{T}_{U}^{\beta}(m)$  in such a manner that $\mathbf{C} =
    \mathbf{B}^{*}\mathbf{B} \in \mathfrak{S}_{m}^{\beta}$,
    \begin{equation}\label{qk5}
        q_{\kappa}(\mathbf{B}^{*}\mathbf{AB}) = q_{\kappa}(\mathbf{C})q_{\kappa}(\mathbf{A})
    \end{equation}
    and
    \begin{equation}\label{qk6}
        q_{\kappa}(\mathbf{B}^{*-1}\mathbf{A}\mathbf{B}^{-1}) = (q_{\kappa}(\mathbf{C}))^{-1}q_{\kappa}(\mathbf{A})
        = q_{-\kappa}(\mathbf{C})q_{\kappa}(\mathbf{A}),
    \end{equation}
see \citet[p. 776, eq. (2.1)]{hlz:08}.
\end{enumerate}
\begin{remark}
Let $\mathcal{P}(\mathfrak{S}_{m}^{\beta})$ denote the algebra of all polynomial functions on
$\mathfrak{S}_{m}^{\beta}$, and $\mathcal{P}_{k}(\mathfrak{S}_{m}^{\beta})$ the subspace of homogeneous
polynomials of degree $k$ and let $\mathcal{P}^{\kappa}(\mathfrak{S}_{m}^{\beta})$ be an irreducible
subspace of $\mathcal{P}(\mathfrak{S}_{m}^{\beta})$ such that
$$
  \mathcal{P}_{k}(\mathfrak{S}_{m}^{\beta}) = \sum_{\kappa}\bigoplus
  \mathcal{P}^{\kappa}(\mathfrak{S}_{m}^{\beta}).
$$
Note that $q_{\kappa}$ is a homogeneous polynomial of degree $k$, moreover $q_{\kappa} \in
\mathcal{P}^{\kappa}(\mathfrak{S}_{m}^{\beta})$, see \cite{gr:87}.
\end{remark}
In (\ref{gammagen1}), $[a]_{\kappa}^{\beta}$ denotes the generalised Pochhammer symbol of weight
$\kappa$, defined as
\begin{eqnarray*}
  [a]_{\kappa}^{\beta} &=& \prod_{i = 1}^{m}(a-(i-1)\beta/2)_{k_{i}}\\
    &=& \frac{\pi^{m(m-1)\beta/4} \displaystyle\prod_{i=1}^{m}
    \Gamma[a + k_{i} -(i-1)\beta/2]}{\Gamma_{m}^{\beta}[a]} \\
    &=& \frac{\Gamma_{m}^{\beta}[a,\kappa]}{\Gamma_{m}^{\beta}[a]},
\end{eqnarray*}
where $\re(a) > (m-1)\beta/2 - k_{m}$ and
$$
  (a)_{i} = a (a+1)\cdots(a+i-1),
$$
is the standard Pochhammer symbol.

An alternative definition of the generalised gamma function of weight $\kappa$ is proposed by
\cite{k:66}, which is defined as%
\begin{eqnarray}\label{int2}
  \Gamma_{m}^{\beta}[a,-\kappa] &=& \displaystyle\int_{\mathbf{A} \in \mathfrak{P}_{m}^{\beta}}
    \etr\{-\mathbf{A}\} |\mathbf{A}|^{a-(m-1)\beta/2 - 1} q_{\kappa}(\mathbf{A}^{-1})
    (d\mathbf{A}) \\
&=& \pi^{m(m-1)\beta/4}\displaystyle\prod_{i=1}^{m} \Gamma[a - k_{i}
    -(m-i)\beta/2] \nonumber\\ \label{gammagen2}
&=& \displaystyle\frac{(-1)^{k} \Gamma_{m}^{\beta}[a]}{[-a +(m-1)\beta/2
    +1]_{\kappa}^{\beta}} ,
\end{eqnarray}
where $\re(a) > (m-1)\beta/2 + k_{1}$.

In addition consider the following generalised beta functions termed, \emph{generalised c-beta function},
see \citet[p. 130]{fk:94} and \cite{dg:15b},
$$
  \mathcal{B}_{m}^{\beta}[a,\kappa;b, \tau] \hspace{10cm}
$$
$$
  \hspace{.5cm}= \int_{\mathbf{0}<\mathbf{S}<\mathbf{I}_{m}}
    |\mathbf{S}|^{a-(m-1)\beta/2-1} q_{\kappa}(\mathbf{S})|\mathbf{I}_{m} - \mathbf{S}|^{b-(m-1)\beta/2-1}
    q_{\tau}(\mathbf{I}_{m} - \mathbf{S})(d\mathbf{S})
$$
$$
   = \int_{\mathbf{R} \in\mathfrak{P}_{m}^{\beta}} |\mathbf{R}|^{a-(m-1)\beta/2-1}
    q_{\kappa}(\mathbf{R})|\mathbf{I}_{m} + \mathbf{R}|^{-(a+b)} q_{-(\kappa+\tau)}(\mathbf{I}_{m} + \mathbf{R})
    (d\mathbf{R})
$$
$$
   = \frac{\Gamma_{m}^{\beta}[a,\kappa] \Gamma_{m}^{\beta}[b,\tau]}{\Gamma_{m}^{\beta}[a+b,
    \kappa+\tau]},\hspace{8.5cm}
$$
where $\kappa = (k_{1}, k_{2}, \dots, k_{m}) \in \Re^{m}$, $\tau = (t_{1}, t_{2}, \dots, t_{m}) \in
\Re^{m}$, Re$(a)> (m-1)\beta/2-k_{m}$ and Re$(b)> (m-1)\beta/2-t_{m}$. Similarly is defined the
\emph{generalised k-beta function} as, see \cite{dg:15b},
$$
  \mathcal{B}_{m}^{\beta}[a,-\kappa;b, -\tau]  \hspace{10cm}
$$
$$
    \hspace{.1cm}=\int_{\mathbf{0}<\mathbf{S}<\mathbf{I}_{m}}
    |\mathbf{S}|^{a-(m-1)\beta/2-1} q_{\kappa}(\mathbf{S}^{-1})|\mathbf{I}_{m} - \mathbf{S}|^{b-(m-1)\beta/2-1}
    q_{\tau}\left((\mathbf{I}_{m} - \mathbf{S})^{-1}\right)(d\mathbf{S})
$$
$$
  =\int_{\mathbf{R} \in\mathfrak{P}_{m}^{\beta}} |\mathbf{R}|^{a-(m-1)\beta/2-1}
    q_{\kappa}(\mathbf{R}^{-1})|\mathbf{I}_{m} + \mathbf{R}|^{-(a+b)} q_{-(\kappa+\tau)}
    \left((\mathbf{I}_{m} + \mathbf{R})^{-1}\right)(d\mathbf{R})\hspace{1cm}
$$
$$
    = \frac{\Gamma_{m}^{\beta}[a,-\kappa] \Gamma_{m}^{\beta}[b,-\tau]}{\Gamma_{m}^{\beta}[a+b,
    -\kappa-\tau]},\hspace{8.5cm}
$$
where $\kappa = (k_{1}, k_{2}, \dots, k_{m}) \in \Re^{m}$, $\tau = (t_{1}, t_{2}, \dots, t_{m}) \in
\Re^{m}$, Re$(a)> (m-1)\beta/2+k_{1}$ and Re$(b)> (m-1)\beta/2+t_{1}$.

Finally, the following Jacobians involving the $\beta$ parameter, reflects the generalised power of the
algebraic technique; the can be seen as extensions of the full derived and unconnected results in the
real, complex or quaternion cases, see \citet{fk:94} and \citet{dggj:11}. These results are the base for
several matrix and matric variate generalised analysis.

\begin{proposition}\label{lemlt}
Let $\mathbf{X}$ and $\mathbf{Y} \in {\mathcal L}_{n,m}^{\beta}$  be matrices of functionally independent
variables, and let $\mathbf{Y} = \mathbf{AXB} + \mathbf{C}$, where $\mathbf{A} \in {\mathcal
L}_{n,n}^{\beta}$, $\mathbf{B} \in {\mathcal L}_{m,m}^{\beta}$ and $\mathbf{C} \in {\mathcal
L}_{n,m}^{\beta}$ are constant matrices. Then
\begin{equation}\label{lt}
    (d\mathbf{Y}) = |\mathbf{A}^{*}\mathbf{A}|^{m\beta/2} |\mathbf{B}^{*}\mathbf{B}|^{
    mn\beta/2}(d\mathbf{X}).
\end{equation}
\end{proposition}

\begin{proposition}\label{lemhlt}
Let $\mathbf{X}$ and $\mathbf{Y} \in \mathfrak{S}_{m}^{\beta}$ be matrices of functionally independent
variables, and let $\mathbf{Y} = \mathbf{AXA^{*}} + \mathbf{C}$, where $\mathbf{A} \in {\mathcal
L}_{m,m}^{\beta}$ and $\mathbf{C} \in \mathfrak{S}_{m}^{\beta}$ are constant matrices. Then
\begin{equation}\label{hlt}
    (d\mathbf{Y}) = |\mathbf{A}^{*}\mathbf{A}|^{(m-1)\beta/2+1} (d\mathbf{X}).
\end{equation}
\end{proposition}

\begin{proposition}\label{lemW}
Let $\mathbf{X} \in {\mathcal L}_{n,m}^{\beta}$  be matrix of functionally independent variables, and
write $\mathbf{X}=\mathbf{V}_{1}\mathbf{T}$, where $\mathbf{V}_{1} \in {\mathcal V}_{m,n}^{\beta}$ and
$\mathbf{T}\in \mathfrak{T}_{U}^{\beta}(m)$ with positive diagonal elements. Define $\mathbf{S} =
\mathbf{X}^{*}\mathbf{X} \in \mathfrak{P}_{m}^{\beta}.$ Then
\begin{equation}\label{w}
    (d\mathbf{X}) = 2^{-m} |\mathbf{S}|^{\beta(n - m + 1)/2 - 1}
    (d\mathbf{S})(\mathbf{V}_{1}^{*}d\mathbf{V}_{1}),
\end{equation}
\end{proposition}

\section{Matricvariate Pearson type II-Riesz distribution}\label{sec3}

Two versions of the matricvariate Pearson type II-Riesz distributions and the corresponding generalised
beta type I distributions are obtained in this section.

A detailed discussion of Riesz distribution may be found in \citet{hl:01} and \citet{dg:15a}. In addition
the Kotz-Riesz distribution is studied in detail in \citet{dg:15b}. For your convenience, we adhere to
standard notation stated in \citet{dg:15a, dg:15b} and consider the following two definitions.

\begin{definition}\label{defEKR} Let $\boldsymbol{\Sigma} \in
\boldsymbol{\Phi}_{m}^{\beta}$, $\boldsymbol{\Theta} \in \boldsymbol{\Phi}_{n}^{\beta}$,
$\boldsymbol{\mu} \in \mathfrak{L}^{\beta}_{n,m}$ and  $\kappa = (k_{1}, k_{2}, \dots, k_{m}) \in
\Re^{m}$. And let $\mathbf{Y} \in \mathfrak{L}^{\beta}_{n,m}$ and $\u(\mathbf{B}) \in
\mathfrak{T}_{U}^{\beta}(n)$, such that $\mathbf{B} = \u(\mathbf{B})^{*}\u(\mathbf{B})$ is the Cholesky
decomposition of $\mathbf{B} \in \mathfrak{S}_{m}^{\beta}$.
\begin{enumerate}
  \item Then it is said that $\mathbf{Y}$ has a Kotz-Riesz distribution of type I and its density function is
  $$
    \frac{\beta^{mn\beta/2+\sum_{i = 1}^{m}k_{i}}\Gamma_{m}^{\beta}[n\beta/2]}{\pi^{mn\beta/2}\Gamma_{m}^{\beta}[n\beta/2,\kappa]
    |\boldsymbol{\Sigma}|^{n\beta/2}|\boldsymbol{\Theta}|^{m\beta/2}}\hspace{4cm}
  $$
  $$\hspace{1cm}
    \times \etr\left\{- \beta\tr \left [\boldsymbol{\Sigma}^{-1} (\mathbf{Y} - \boldsymbol{\mu})^{*}
    \boldsymbol{\Theta}^{-1}(\mathbf{Y} - \boldsymbol{\mu})\right ]\right\}
  $$
  \begin{equation}\label{dfEKR1}\hspace{3.1cm}
    \times q_{\kappa}\left [\u(\boldsymbol{\Sigma})^{*-1} (\mathbf{Y} - \boldsymbol{\mu})^{*}
    \boldsymbol{\Theta}^{-1}(\mathbf{Y} - \boldsymbol{\mu})\u(\boldsymbol{\Sigma})^{-1}\right ](d\mathbf{Y})
  \end{equation}
  with $\re(n\beta/2) > (m-1)\beta/2 - k_{m}$;  denoting this fact as
  $$
    \mathbf{Y} \sim \mathcal{K}\mathcal{R}^{\beta, I}_{n \times m}
    (\kappa,\boldsymbol{\mu}, \boldsymbol{\Theta}, \boldsymbol{\Sigma}).
  $$
  \item Then it is said that $\mathbf{Y}$ has a Kotz-Riesz distribution of type II and its density function is
  $$
    \frac{\beta^{mn\beta/2-\sum_{i = 1}^{m}k_{i}}\Gamma_{m}^{\beta}[n\beta/2]}{\pi^{mn\beta/2}\Gamma_{m}^{\beta}[n\beta/2,-\kappa]
    |\boldsymbol{\Sigma}|^{n\beta/2}|\boldsymbol{\Theta}|^{m\beta/2}}\hspace{4cm}
  $$
  $$
    \times \etr\left\{- \beta\tr \left [\boldsymbol{\Sigma}^{-1} (\mathbf{Y} - \boldsymbol{\mu})^{*}
    \boldsymbol{\Theta}^{-1}(\mathbf{Y} - \boldsymbol{\mu})\right ]\right\}
  $$
  \begin{equation}\label{dfEKR2}\hspace{2.5cm}
    \times q_{\kappa}\left [\left(\u(\boldsymbol{\Sigma})^{*-1} (\mathbf{Y} - \boldsymbol{\mu})^{*}
    \boldsymbol{\Theta}^{-1}(\mathbf{Y} - \boldsymbol{\mu})\u(\boldsymbol{\Sigma})^{-1/2}\right)^{-1}\right ](d\mathbf{Y})
  \end{equation}
  with $\re(n\beta/2) > (m-1)\beta/2 + k_{1}$;  denoting this fact as
  $$
    \mathbf{Y} \sim \mathcal{KR}^{\beta, II}_{n \times m}
    (\kappa,\boldsymbol{\mu}, \boldsymbol{\Theta}, \boldsymbol{\Sigma}).
  $$
\end{enumerate}
\end{definition}

\begin{definition}\label{defnRd} Let $\mathbf{\Xi} \in
\mathbf{\Phi}_{m}^{\beta}$ and  $\kappa = (k_{1}, k_{2}, \dots, k_{m}) \in \Re^{m}$.
\begin{enumerate}
  \item Then it is said that $\mathbf{V}$ has a Riesz distribution of type I if its density function is
  \begin{equation}\label{dfR1}
    \frac{\beta^{am+\sum_{i = 1}^{m}k_{i}}}{\Gamma_{m}^{\beta}[a,\kappa] |\mathbf{\Xi}|^{a}q_{\kappa}(\mathbf{\Xi})}
    \etr\{-\beta\mathbf{\Xi}^{-1}\mathbf{V}\}|\mathbf{V}|^{a-(m-1)\beta/2 - 1}
    q_{\kappa}(\mathbf{V})(d\mathbf{V})
  \end{equation}
  for $\mathbf{V} \in \mathfrak{P}_{m}^{\beta}$ and $\re(a) \geq (m-1)\beta/2 - k_{m}$;
  denoting this fact as $\mathbf{V} \sim \mathcal{R}^{\beta, I}_{m}(a,\kappa,
  \mathbf{\Xi})$.
  \item Then it is said that $\mathbf{V}$ has a Riesz distribution of type II if its density function is
  \begin{equation}\label{dfR2}
     \frac{\beta^{am-\sum_{i = 1}^{m}k_{i}}}{\Gamma_{m}^{\beta}[a,-\kappa]
   |\mathbf{\Xi}|^{a}q_{\kappa}(\mathbf{\Xi}^{-1})}\etr\{-\beta\mathbf{\Xi}^{-1}\mathbf{V}\}
  |\mathbf{V}|^{a-(m-1)\beta/2 - 1} q_{\kappa}(\mathbf{V}^{-1}) (d\mathbf{V})
  \end{equation}
  for $\mathbf{V} \in \mathfrak{P}_{m}^{\beta}$ and $\re(a) > (m-1)\beta/2 + k_{1}$;
  denoting this fact as $\mathbf{V} \sim \mathcal{R}^{\beta, II}_{m}(a,\kappa,
  \mathbf{\Xi})$.
\end{enumerate}
\end{definition}

\begin{theorem}\label{teo1}
Let $\kappa = (k_{1}, k_{2}, \dots, k_{m}) \in \Re^{m}$, and $\tau = (t_{1}, t_{2}, \dots, t_{m}) \in
\Re^{m}$. Also define $\mathbf{R}\in {\mathcal L}_{n,m}^{\beta}$ as
$$
  \mathbf{R} = \mathbf{X}\mathbf{L}^{-1},
$$
where $\mathbf{L} \in \mathfrak{T}_{U}^{\beta}(m)$ is such that $\mathbf{U}=\mathbf{L}^{*}\mathbf{L} =
\mathbf{U}_{1} + \mathbf{X}^{*}\mathbf{X}$  is the Cholesky decomposition of $\mathbf{U}$,
\begin{enumerate}
  \item with $\mathbf{U}_{1}\sim \mathcal{R}_{m}^{\beta,I}(\nu\beta/2,
    \kappa,\mathbf{I}_{m})$,  $\re(\nu\beta/2)> (m-1)\beta/2-k_{m}$; independent of $\mathbf{X}
    \sim \mathcal{KR}_{n \times m}^{\beta,I}(\tau,\mathbf{0}, \mathbf{I}_{n}, \mathbf{I}_{m})$,
    $\re(n\beta/2)> (m-1)\beta/2-t_{m}$. Then $\mathbf{U} \sim \mathcal{R}_{m}^{\beta,I}((\nu +
    n)\beta/2, \kappa+\tau,\mathbf{I}_{m})$ independent of $\mathbf{R}$ with $\re((\nu+n)\beta/2)>
    (m-1)\beta/2-k_{m}-t_{m}$. Furthermore, the density of $\mathbf{R}$ is
    \begin{equation}\label{DR1I}
        \frac{\Gamma_{m}^{\beta}[n\beta/2]\quad\left|\mathbf{I}_{m} - \mathbf{R}^{*}\mathbf{R}\right|^{(\nu-m+1)\beta/2-1}}
        {\pi^{mn\beta/2}\mathcal{B}_{m}^{\beta}[\nu\beta /2,\kappa;n\beta/2,\tau]} \
        q_{\kappa}\left(\mathbf{I}_{m} - \mathbf{R}^{*}\mathbf{R}\right)
        q_{\tau}\left(\mathbf{R}^{*}\mathbf{R}\right)(d\mathbf{R}),
    \end{equation}
    which is shall be termed the \emph{matricvariate Pearson type II-Riesz distribution type I},
    where $\mathbf{I}_{m} - \mathbf{R}^{*}\mathbf{R} \in \mathfrak{P}^{\beta}_{m}$.

  \item with $\mathbf{U}_{1}\sim \mathcal{R}_{m}^{\beta,II}(\nu\beta/2,
    \kappa,\mathbf{I}_{m})$,  $\re(\nu\beta/2)> (m-1)\beta/2+k_{1}$; independent of $\mathbf{X}
    \sim \mathcal{KR}_{n \times m}^{\beta,II}(\tau,\mathbf{0}, \mathbf{I}_{n}, \mathbf{I}_{m})$,
    $\re(n\beta/2)> (m-1)\beta/2+t_{1}$. Then $\mathbf{U} \sim \mathcal{R}_{m}^{\beta,II}((\nu +
    n)\beta/2, \kappa+\tau,\mathbf{I}_{m})$ independent of $\mathbf{R}$ with $\re((\nu+n)\beta/2)>
    (m-1)\beta/2+k_{1}+t_{1}$. Furthermore, the density of $\mathbf{R}$ is
    \begin{equation*}
      \frac{\Gamma_{m}^{\beta}[n\beta/2]\quad\left|\mathbf{I}_{m} - \mathbf{R}^{*}\mathbf{R}\right|^{(\nu-m+1)\beta/2-1}}
      {\pi^{mn\beta/2}\mathcal{B}_{m}^{\beta}[\nu\beta /2,-\kappa;n\beta/2,-\tau]} \
      q_{\kappa}\left[\left(\mathbf{I}_{m} - \mathbf{R}^{*}\mathbf{R}\right)^{-1}\right]\hspace{2cm}
    \end{equation*}
    \begin{equation}\label{DR1II}
      \hspace{7cm}\times \ q_{\tau}\left[\left(\mathbf{R}^{*}\mathbf{R}\right)^{-1}\right](d\mathbf{R}),
    \end{equation}
    which is shall be termed the \emph{matricvariate Pearson type II-Riesz distribution type II},
    where $\mathbf{I}_{m} - \mathbf{R}^{*}\mathbf{R} \in \mathfrak{P}^{\beta}_{m}$.
\end{enumerate}
\end{theorem}
\begin{proof} 1. From definitions \ref{defEKR} and \ref{defnRd}, the joint density of $\mathbf{U}_{1}$ and $\mathbf{X}$ is
$$
  \propto |\mathbf{U}_{1}|^{(\nu-m+1)\beta/2-1}\etr\{-\beta\left(
  \mathbf{U}_{1} + \mathbf{X}^{*}\mathbf{X}\right)\}
 q_{\kappa}(\mathbf{U}_{1})q_{\tau}\left(\mathbf{X}^{*}\mathbf{X}
  \right)(d\mathbf{U}_{1})(d\mathbf{X}),
$$
where the constant of proportionality given by
$$
  c = \frac{\beta^{\nu m\beta/2+\sum_{i=1}^{m}k_{i}}}{\Gamma_{m}^{\beta}[\nu\beta /2,\kappa]} \ \cdot \
  \frac{\beta^{mn\beta/2+\sum_{i=1}^{m}t_{i}}\Gamma_{m}^{\beta}[n\beta/2]}{\pi^{mn\beta/2}
  \Gamma_{m}^{\beta}[n\beta /2,\tau] }.
$$
Making the change of variable $\mathbf{U}_{1} = (\mathbf{U} -\mathbf{X}^{*}\mathbf{X})$ and $\mathbf{X} =
\mathbf{R}\mathbf{L}$, where $\mathbf{U} = \mathbf{L}^{*}\mathbf{L}$, then by (\ref{lt})
$$
  (d\mathbf{U}_{1})(d\mathbf{X}) = |\mathbf{L}^{*}\mathbf{L}|^{n\beta/2}(d\mathbf{U})(d\mathbf{R})
    = |\mathbf{U}|^{n\beta/2}(d\mathbf{U})(d\mathbf{R}),
$$
and observing that $|\mathbf{U}_{1}| = |\mathbf{U} - \mathbf{X}^{*}\mathbf{X}| = |\mathbf{U} -
\mathbf{L}^{*}\mathbf{R}^{*}\mathbf{RL}| = |\mathbf{U}||\mathbf{I}_{m} - \mathbf{R}^{*}\mathbf{R}|$, the
joint density of $\mathbf{U}$ and $\mathbf{R}$ is
$$
  \propto |\mathbf{U}|^{(\nu+n-m+1)\beta/2-1}\etr\left\{-\beta\mathbf{U}\right\} q_{\kappa+\tau}(\mathbf{U})
  |\mathbf{I}_{m}-\mathbf{R}^{*}\mathbf{R}|^{(\nu-m+1)\beta/2-1}
$$
$$
  \times q_{\kappa}(\mathbf{I}_{m}-\mathbf{R}^{*}\mathbf{R})\ q_{\tau}\left(\mathbf{R}^{*}
  \mathbf{R}\right)(d\mathbf{U})(d\mathbf{R}).
$$
Finally, note that the joint density of $\mathbf{U}$ and $\mathbf{R}$ is
$$
  = \frac{\beta^{(\nu+n) m\beta/2+\sum_{i=1}^{m}(k_{i}+t_{i})}}{\Gamma_{m}^{\beta}[(\nu+n)\beta /2,\kappa+\tau]}
  |\mathbf{U}|^{(\nu+n-m+1)\beta/2-1}\etr\left\{-\beta\mathbf{U}\right\}
  q_{\kappa+\tau}(\mathbf{U})(d\mathbf{U})
$$
$$
  \times \ \frac{\Gamma_{m}^{\beta}[\nu\beta /2,\kappa] \ |\mathbf{I}_{m}-\mathbf{R}^{*}\mathbf{R}|^{(\nu-m+1)\beta/2-1}}
  {\pi^{mn\beta/2}\mathcal{B}_{m}^{\beta}[\nu\beta /2, \kappa;n\beta/2,\tau]}
  q_{\kappa}(\mathbf{I}_{m}-\mathbf{R}^{*}\mathbf{R})\ q_{\tau}\left(\mathbf{R}^{*} \mathbf{R}\right) (d\mathbf{R})
$$
which shows that $\mathbf{U} \sim \mathcal{R}_{m}^{\beta,I}((\nu + n)\beta/2,
\kappa+\tau,\mathbf{I}_{m})$ and is independent of $\mathbf{R}$.

2. Its proof is similar to given for item 1.\qed
\end{proof}

An alternative way to define the matricvariate Pearson typeII-Riesz distributions is collected in the
following result.

\begin{corollary}\label{cor0}
Let $\kappa_{1} = (k_{11}, k_{12}, \dots, k_{1n})\in \Re^{n}$, and $\tau_{1} = (t_{11}, t_{12}, \dots,
t_{1n})\in \Re^{n}$. Also define $\mathbf{R}_{1}\in {\mathcal L}_{n,m}^{\beta}$ as
$$
  \mathbf{R}_{1} = \mathbf{L}_{1}^{-1}\mathbf{Y},
$$
with $\mathbf{L}_{1}^{*}\in \mathfrak{T}_{U}^{\beta}(n)$ is such that $\mathbf{V}=\mathbf{LL}^{*} =
\mathbf{V}_{1} + \mathbf{YY}^{*}$  is the Cholesky decomposition of $\mathbf{V}$,
\begin{enumerate}
  \item  where $\mathbf{V}_{1}\sim \mathcal{R}_{n}^{\beta,I}(a\beta/2,
    \kappa_{1},\mathbf{I}_{n})$,  $\re(a\beta/2)> (n-1)\beta/2-k_{1n}$; independent of $\mathbf{Y} = \mathbf{X}^{*}
    \sim \mathcal{KR}_{n \times m}^{\beta,I}(\tau_{1},\mathbf{0}, \mathbf{I}_{n}, \mathbf{I}_{m})$,
    $\re(m\beta/2)> (n-1)\beta/2-t_{1n}$. Then $\mathbf{U} \sim \mathcal{R}_{n}^{\beta,I}((a +
    m)\beta/2, \kappa_{1}+\tau_{1},\mathbf{I}_{n})$ independent of $\mathbf{R}$ with $\re((a+m)\beta/2)>
    (n-1)\beta/2-k_{1n}-t_{1n}$. Furthermore, the density of $\mathbf{R}$ is
    \begin{equation}\label{DR1I0}
        \frac{\Gamma_{n}^{\beta}[m\beta/2]\quad\left|\mathbf{I}_{n} - \mathbf{R}_{1}\mathbf{R}_{1}^{*}\right|^{(a-n+1)\beta/2-1}}
        {\pi^{mn\beta/2}\mathcal{B}_{n}^{\beta}[a\beta /2,\kappa_{1};n\beta/2,\tau_{1}]}
        q_{\kappa_{1}}\left(\mathbf{I}_{n} - \mathbf{R}_{1}\mathbf{R}_{1}^{*}\right)
        q_{\tau_{1}}\left(\mathbf{R}_{1}\mathbf{R}_{1}^{*}\right)(d\mathbf{R}_{1}),
    \end{equation}
    which is shall be termed the \emph{matricvariate Pearson type II-Riesz distribution type I},
    where $\mathbf{I}_{m} - \mathbf{R}_{1}\mathbf{R}_{1}^{*} \in \mathfrak{P}^{\beta}_{m}$.

  \item where $\mathbf{V}_{1}\sim \mathcal{R}_{n}^{\beta,II}(a\beta/2,
    \kappa_{1},\mathbf{I}_{n})$,  $\re(\nu\beta/2)> (n-1)\beta/2+k_{11}$; independent of $\mathbf{Y} = \mathbf{X}^{*}
    \sim \mathcal{KR}_{n \times m}^{\beta,II}(\tau_{1},\mathbf{0}, \mathbf{I}_{n}, \mathbf{I}_{m})$,
    $\re(m\beta/2)> (n-1)\beta/2+t_{11}$. Then $\mathbf{V} \sim \mathcal{R}_{n}^{\beta,II}((a +
    m)\beta/2, \kappa_{1}+\tau_{1},\mathbf{I}_{n})$ independent of $\mathbf{R}$ with $\re((a+m)\beta/2)>
    (n-1)\beta/2+k_{11}+t_{11}$. Furthermore, the density of $\mathbf{R}$ is
    $$
      \frac{\Gamma_{n}^{\beta}[m\beta/2]\quad\left|\mathbf{I}_{n} - \mathbf{R}_{1}\mathbf{R}_{1}^{*}\right|^{(a-n+1)\beta/2-1}}
      {\pi^{mn\beta/2}\mathcal{B}_{n}^{\beta}[a\beta /2,-\kappa_{1};n\beta/2,-\tau_{1}]}
      q_{\kappa_{1}}\left[\left(\mathbf{I}_{n} - \mathbf{R}_{1}\mathbf{R}_{1}^{*}\right)^{-1}\right]
    $$
    \begin{equation}\label{DR1II0}
      \hspace{5cm} \times \ q_{\tau_{1}}\left[\left(\mathbf{R}_{1}\mathbf{R}_{1}^{*}\right)^{-1}\right](d\mathbf{R}_{1}),
    \end{equation}
    which is shall be termed the \emph{matricvariate Pearson type II-Riesz distribution type II},
    where $\mathbf{I}_{n} - \mathbf{R}_{1}\mathbf{R}_{1}^{*} \in \mathfrak{P}^{\beta}_{n}$.
\end{enumerate}
\end{corollary}
\begin{proof}
The proof is a verbatim copy of the proof of Theorem \ref{teo1}. Alternatively, observe that densities
(\ref{DR1I0}) and (\ref{DR1II0}) can be obtained from densities (\ref{DR1I}) and (\ref{DR1II}),
respectively, making the following substitutions,
\begin{equation}\label{s}
   \mathbf{R}\rightarrow \mathbf{R}_{1}^{*}  \quad m \rightarrow n, \quad n \rightarrow m, \quad
   \nu \rightarrow  a,
\end{equation}
and thus, $\kappa \rightarrow \kappa_{1}, \quad \tau \rightarrow \tau_{1},$ and $k_{i} \rightarrow k_{1i}
\quad t_{i} \rightarrow t_{1i}$.
\end{proof}

\begin{corollary}\label{cor11}
Let $\mathbf{Q} = \u(\mathbf{\Omega})^{-1}\mathbf{R}\u(\mathbf{\Xi}) + \boldsymbol{\mu}$, $\mathbf{R}$ as
in Theorem \ref{teo1}, and $\u(\mathbf{\Omega}) \in \mathfrak{T}_{U}^{\beta}(n)$ and $\u(\mathbf{\Xi})
\in \mathfrak{T}_{U}^{\beta}(m)$ are constant matrices such that $\mathbf{\Omega} =
\u(\mathbf{\Omega})^{*}\u(\mathbf{\Omega})\in \mathfrak{P}_{m}^{\beta}$ and $\mathbf{\Xi}=
\u(\mathbf{\Xi})^{*}\u(\mathbf{\Xi}) \in \mathfrak{P}_{n}^{\beta}$, respectively, and $\boldsymbol{\mu}
\in \mathcal{L}_{m,n}^{\beta}$ is constant.

\begin{enumerate}
  \item Then, from (\ref{DR1I}) the density of $\mathbf{Q}$ is
$$
  \propto
  \left|\mathbf{\Xi} - (\mathbf{Q}-\boldsymbol{\mu})^{*}\mathbf{\Omega}(\mathbf{Q}-
  \boldsymbol{\mu})\right|^{(\nu-m+1)\beta/2-1}\hspace{5cm}
$$
\par\noindent\hfill\mbox{$\times \ q_{\kappa}\left[\mathbf{\Xi} -
(\mathbf{Q}-\boldsymbol{\mu})^{*}\mathbf{\Omega}(\mathbf{Q}- \boldsymbol{\mu})\right]
q_{\tau}\left[(\mathbf{Q}-\boldsymbol{\mu})^{*}\mathbf{\Omega}(\mathbf{Q}-
  \boldsymbol{\mu})\right](d\mathbf{Q})$,}\par\noindent %
with constant of proportionality
$$
  \frac{\Gamma_{m}^{\beta}[n\beta/2] |\mathbf{\Omega}|^{m\beta/2} }
  {\pi^{mn\beta/2} \mathcal{B}_{m}^{\beta}[\nu\beta/2,\kappa;n\beta/2, \tau]
  |\mathbf{\Xi}|^{(\nu+n-m+1)\beta/2-1} q_{\kappa+\tau}(\mathbf{\Xi})}
$$
where $\mathbf{\Xi} - (\mathbf{Q}-\boldsymbol{\mu})^{*}\mathbf{\Omega}(\mathbf{Q}- \boldsymbol{\mu}) \in
\mathfrak{P}_{m}^{\beta}$. This fact is denoted as
$$
  \mathbf{Q} \sim \mathcal{P_{II}R}_{n \times m}^{\beta,I}(\nu,\kappa, \tau, \boldsymbol{\mu},
  \mathbf{\Omega}, \mathbf{\Xi}).
$$
\item And from (\ref{DR1II}) the density of $\mathbf{Q}$ is
$$
  \propto
  \left|\mathbf{\Xi} - (\mathbf{Q}-\boldsymbol{\mu})^{*}\mathbf{\Omega}(\mathbf{Q}-
  \boldsymbol{\mu})\right|^{(\nu-m+1)\beta/2-1}\hspace{5cm}
$$
\par\noindent\hfill\mbox{$\times \ q_{\kappa}\left[\left(\mathbf{\Xi} -
(\mathbf{Q}-\boldsymbol{\mu})^{*}\mathbf{\Omega}(\mathbf{Q}- \boldsymbol{\mu})\right)^{-1}\right]
q_{\tau}\left[\left((\mathbf{Q}-\boldsymbol{\mu})^{*}\mathbf{\Omega}(\mathbf{Q}-
  \boldsymbol{\mu})\right)^{-1}\right](d\mathbf{Q})$,}\par\noindent %
with constant of proportionality
$$
  \frac{\Gamma_{m}^{\beta}[n\beta/2] |\mathbf{\Omega}|^{m\beta/2} }
  {\pi^{mn\beta/2} \mathcal{B}_{m}^{\beta}[\nu\beta/2,-\kappa;n\beta/2, -\tau]
  |\mathbf{\Xi}|^{(\nu+n-m+1)\beta/2-1} q_{\kappa+\tau}\left(\mathbf{\Xi}^{-1}\right)}
$$
where $\mathbf{\Xi} - (\mathbf{Q}-\boldsymbol{\mu})^{*}\mathbf{\Omega}(\mathbf{Q}-
  \boldsymbol{\mu}) \in \mathfrak{P}_{m}^{\beta}$.This fact is denoted as
$$
  \mathbf{Q} \sim \mathcal{P_{II}R}_{m \times n}^{\beta,II}(\nu,\kappa, \tau, \boldsymbol{\mu},
  \mathbf{\Omega}, \mathbf{\Xi}).
$$
\end{enumerate}
\end{corollary}
\begin{proof} 1. The proof follows from (\ref{DR1I}) and (\ref{DR1II}), respectively, observing that, by
(\ref{lt})
$$
  (d\mathbf{R}) = |\mathbf{\Omega}|^{m\beta/2}|\mathbf{\Xi}|^{-n\beta/2}(d\mathbf{Q}),
$$
and
\begin{eqnarray*}
  (\mathbf{I}_{m} - \mathbf{R}^{*}\mathbf{R}) &=& (\mathbf{I}_{m}- \u(\mathbf{\Xi})^{*-1}
  (\mathbf{Q}-\boldsymbol{\mu})^{*}\u(\mathbf{\Omega})^{*}\u(\mathbf{\Omega})
  (\mathbf{Q}-\boldsymbol{\mu})\u(\mathbf{\Xi})^{-1}) \\
   &=& \u(\mathbf{\Xi})^{*-1}(\mathbf{\Xi}- (\mathbf{Q}-\boldsymbol{\mu})^{*}\mathbf{\Omega}
  (\mathbf{Q}-\boldsymbol{\mu}))\u(\mathbf{\Xi})^{-1}.
\end{eqnarray*}
2. This is similar to the given to item 1.  \qed
\end{proof}

Next some basic properties of the matricvariate Pearson type II-Riesz distributions are studied.

\begin{corollary}\label{cor12}
Let $\mathbf{Q}_{1} = \u(\mathbf{\Omega})\mathbf{R}\u(\mathbf{\Xi})^{-1} + \boldsymbol{\mu}$,
$\mathbf{R}$ as in Corollary \ref{cor0}, and $\u(\mathbf{\Omega})^{*} \in \mathfrak{T}_{U}^{\beta}(n)$
and $\u(\mathbf{\Xi})^{*} \in \mathfrak{T}_{U}^{\beta}(m)$ are constant matrices such that
$\mathbf{\Omega} = \u(\mathbf{\Omega})\u(\mathbf{\Omega})^{*}\in \mathfrak{P}_{m}^{\beta}$ and
$\mathbf{\Xi}= \u(\mathbf{\Xi})\u(\mathbf{\Xi})^{*} \in \mathfrak{P}_{n}^{\beta}$, respectively, and
$\boldsymbol{\mu} \in \mathcal{L}_{m,n}^{\beta}$ is constant.

\begin{enumerate}
\item From (\ref{DR1I0}) the density of $\mathbf{Q}_{1}$ is
$$
  \propto
  \left|\mathbf{\Omega} - (\mathbf{Q}_{1}-\boldsymbol{\mu})\mathbf{\Xi}(\mathbf{Q}_{1}-
  \boldsymbol{\mu})^{*}\right|^{(a-n+1)\beta/2-1} q_{\kappa_{1}}\left[\mathbf{\Omega} -
  (\mathbf{Q}_{1}-\boldsymbol{\mu})\mathbf{\Xi}(\mathbf{Q}_{1}- \boldsymbol{\mu})^{*}\right]
$$
\par\noindent\hfill\mbox{$\times \ q_{\tau_{1}}\left[(\mathbf{Q}_{1}-\boldsymbol{\mu})\mathbf{\Xi}(\mathbf{Q}_{1}-
  \boldsymbol{\mu})^{*}\right](d\mathbf{Q}_{1})$,}\par\noindent %
with constant of proportionality
$$
  \frac{\Gamma_{n}^{\beta}[m\beta/2] |\mathbf{\Xi}|^{n\beta/2}}
  {\pi^{mn\beta/2} \mathcal{B}_{n}^{\beta}[a\beta/2,\kappa_{1};m\beta/2, \tau_{1}]
  |\mathbf{\Omega}|^{(a+m-n+1)\beta/2-1}q_{\kappa_{1}+\tau_{1}}(\mathbf{\Omega})}
$$
where $\mathbf{\Omega} - (\mathbf{Q}_{1}-\boldsymbol{\mu})\mathbf{\Xi}(\mathbf{Q}_{1}-
\boldsymbol{\mu})^{*}) \in \mathfrak{P}_{n}^{\beta}$.This fact is denoted as
$$
  \mathbf{Q}_{1} \sim \mathcal{P_{II}R}_{n \times m}^{\beta,I}(a,\kappa_{1}, \tau_{1}, \boldsymbol{\mu},
  \mathbf{\Omega}, \mathbf{\Xi}).
$$
\item Similarly, from (\ref{DR1II0}) the density of $\mathbf{Q}_{1}$ is
$$
  \propto
  \left|\mathbf{\Omega} - (\mathbf{Q}_{1}-\boldsymbol{\mu})\mathbf{\Xi}(\mathbf{Q}_{1}-
  \boldsymbol{\mu})^{*}\right|^{(a-n+1)\beta/2-1} q_{\kappa_{1}}\left[\left(\mathbf{\Omega} -
  (\mathbf{Q}_{1}-\boldsymbol{\mu})\mathbf{\Xi}(\mathbf{Q}_{1}- \boldsymbol{\mu})^{*}\right)^{-1}\right]
$$
\par\noindent\hfill\mbox{$\times \ q_{\tau_{1}}\left[\left((\mathbf{Q}_{1}-\boldsymbol{\mu})
\mathbf{\Xi}(\mathbf{Q}_{1}- \boldsymbol{\mu})^{*}\right)^{-1}\right](d\mathbf{Q}_{1})$,}\par\noindent %
with constant of proportionality
$$
  \frac{\Gamma_{n}^{\beta}[m\beta/2] |\mathbf{\Xi}|^{n\beta/2}}
  {\pi^{mn\beta/2} \mathcal{B}_{n}^{\beta}[a\beta/2,-\kappa_{1};m\beta/2, -\tau_{1}]
  |\mathbf{\Omega}|^{(a+m-n+1)\beta/2-1}q_{\kappa_{1}+\tau_{1}}(\mathbf{\Omega}^{-1})}
$$
where $\mathbf{\Omega} - (\mathbf{Q}_{1}-\boldsymbol{\mu})\mathbf{\Xi}(\mathbf{Q}_{1}-
\boldsymbol{\mu})^{*}) \in \mathfrak{P}_{n}^{\beta}$. This fact is denoted as
$$
  \mathbf{Q}_{1} \sim \mathcal{P_{II}R}_{m \times n}^{\beta,II}(a,\kappa, \tau, \boldsymbol{\mu},
  \mathbf{\Omega}, \mathbf{\Xi}).
$$
\end{enumerate}
\end{corollary}
\begin{proof} 1. The proof follows from (\ref{DR1I0}) and (\ref{DR1II0}), respectively, observing that, by
(\ref{lt})
$$
  (d\mathbf{R}_{1}) = |\mathbf{\Omega}|^{-m\beta/2}|\mathbf{\Xi}|^{n\beta/2}(d\mathbf{Q}_{1}),
$$
and
\begin{eqnarray*}
  (\mathbf{I}_{n} - \mathbf{R}_{1}\mathbf{R}_{1}^{*}) &=& (\mathbf{I}_{m}- \u(\mathbf{\Omega})^{-1}
  (\mathbf{Q}_{1}-\boldsymbol{\mu})\u(\mathbf{\Xi})\u(\mathbf{\Xi})^{*}
  (\mathbf{Q}_{1}-\boldsymbol{\mu})^{*}\u(\mathbf{\Omega})^{*-1}) \\
   &=& \u(\mathbf{\Omega})^{-1}(\mathbf{\Omega}- (\mathbf{Q}_{1}-\boldsymbol{\mu})\mathbf{\Xi}
  (\mathbf{Q}_{1}-\boldsymbol{\mu})^{*})\u(\mathbf{\Omega})^{*-1}.
\end{eqnarray*}
2. This is similar to the given to item 1.  \qed
\end{proof}

Now c-beta-Riesz type I and k-beta-Riesz type I distributions are obtained, see \cite{dg:15b}. Let $n
\geq m$ and let $\mathbf{B} \in \mathfrak{P}_{m}^{\beta}$ defined as $\mathbf{B} =
\mathbf{R}^{*}\mathbf{R}$ then, under the conditions of Theorem \ref{teo1}, we have
$$
  \mathbf{B} = \mathbf{R}^{*}\mathbf{R} =\mathbf{L}^{*-1}\mathbf{X}^{*}\mathbf{X}\mathbf{L}^{-1} =
  \mathbf{L}^{*-1}\mathbf{W}\mathbf{L}^{-1}
$$
where $\mathbf{W} = \mathbf{X}^{*}\mathbf{X}$ and $\mathbf{L} \in \mathfrak{T}_{U}^{\beta}(m)$ is such
that $\mathbf{U}=\mathbf{L}^{*}\mathbf{L} = \mathbf{U}_{1} + \mathbf{X}^{*}\mathbf{X}$ is the Cholesky
decomposition of $\mathbf{U}$. Therefore:
\begin{theorem}\label{teo2}
\begin{enumerate}
  \item Assuming that $\mathbf{R} \sim \mathcal{P_{II}R}_{n \times m}^{\beta,I}(\nu,\kappa, \tau,
    \boldsymbol{0}, \textbf{I}_{n},\textbf{I}_{m})$. Then, the density of $\mathbf{B}$, such that $\mathbf{I}_{m}-\mathbf{B} \in
    \mathfrak{P}_{m}^{\beta}$ is
    \begin{equation}\label{BI1}
        \frac{|\mathbf{B}|^{(n-m+1)\beta/2-1}}{\mathcal{B}_{m}^{\beta}[\nu\beta/2, \kappa;n\beta/2,\tau]}
        |\mathbf{I}_{m}-\mathbf{B}|^{(\nu-m+1)\beta/2-1}q_{\kappa}(\mathbf{I}_{m}-\mathbf{B})q_{\tau}(\mathbf{B})(d\mathbf{B}).
    \end{equation}
    $\mathbf{B}$ is said to have a \emph{matricvariate c-beta-Riesz type I distribution}.

  \item Suppose that $\ \mathbf{R} \ \sim \mathcal{P_{II}R}_{n \times m}^{\beta,II}(\nu,\kappa, \tau,
    \boldsymbol{0}, \textbf{I}_{n},\textbf{I}_{m})$. Then the density of $\mathbf{B}$,
    such that $\mathbf{I}_{m}-\mathbf{B} \in \mathfrak{P}_{m}^{\beta}$ is
    \begin{equation}\label{BII1}
        \frac{|\mathbf{B}|^{(n-m+1)\beta/2-1}}{\mathcal{B}_{m}^{\beta}[\nu\beta/2, -\kappa;n\beta/2,-\tau]}
        |\mathbf{I}_{m}-\mathbf{B}|^{(\nu-m+1)\beta/2-1}q_{\kappa}[(\mathbf{I}_{m}-\mathbf{B})^{-1}]
    \end{equation}
\par\noindent\hfill\mbox{$\times \ q_{\tau}[(\mathbf{B})^{-1}](d\mathbf{B}).$}\par\noindent %
    $\mathbf{B}$ is said to have a \emph{matricvariate k-beta-Riesz type I distribution}.
\end{enumerate}
\end{theorem}
\begin{proof} 1.From (\ref{DR1I}) the density function of $\mathbf{R}$ is
$$
  \propto \left|\mathbf{I}_{m} - \mathbf{R}^{*}\mathbf{R}\right|^{(\nu-m+1)\beta/2-1}
        q_{\kappa}\left(\mathbf{I}_{m} - \mathbf{R}^{*}\mathbf{R}\right)
        q_{\tau}\left(\mathbf{R}^{*}\mathbf{R}\right)(d\mathbf{R}).
$$
Now make the change of variable $\mathbf{B} = \mathbf{R}^{*}\mathbf{R}$, so that
$$
  (d\mathbf{R}) = 2^{-m} |\mathbf{B}|^{(n - m + 1)\beta/2 - 1}
    (d\mathbf{B})(\mathbf{V}_{1}^{*}d\mathbf{V}_{1}),
$$
with  $\mathbf{V}_{1} \in {\mathcal V}_{m,n}^{\beta}$. The joint density of $\mathbf{B}$ and
$\mathbf{V}_{1}$ is then
$$
  \propto \left|\mathbf{I}_{m} - \mathbf{B}\right|^{(\nu-m+1)\beta/2-1}
        q_{\kappa}\left(\mathbf{I}_{m} - \mathbf{B}\right)
        q_{\tau}\left(\mathbf{B}\right)|\mathbf{B}|^{(n - m + 1)\beta/2 - 1}
        (d\mathbf{R})(\mathbf{V}_{1}^{*}d\mathbf{V}_{1}).
$$
Integrating with respect to $\mathbf{V}_{1}$ using (\ref{vol}), gives the stated marginal density of
$\mathbf{B}$.

2. This is obtained in a similar way to the gives in item 1. \qed
\end{proof}

In addition, assume that $n < m$ and let  $\mathbf{B}_{1} \in \mathfrak{P}_{n}^{\beta}$ defined as
$\mathbf{B}_{1} = \mathbf{R}_{1}\mathbf{R}_{1}^{*}$ then, under the conditions of Corollary \ref{cor0} we
have
$$
  \widetilde{\mathbf{B}} =\mathbf{L}_{1}^{-1}\mathbf{Y}\mathbf{Y}^{*}\mathbf{L}_{1}^{*-1} =
   \mathbf{L}_{1}^{-1}\mathbf{W}_{1}\mathbf{L}_{1}^{*-1},
$$
where $\mathbf{W}_{1} = \mathbf{Y}\mathbf{Y}^{*}$. Hence:
\begin{theorem}\label{teo3}
\begin{enumerate}
  \item  Assuming that $\mathbf{R} \sim \mathcal{P_{II}R}_{n \times m}^{\beta,I}(a,\kappa_{1}, \tau_{1},
    \boldsymbol{0}, \textbf{I}_{n},\textbf{I}_{m})$. Then, the density of  $\mathbf{B}_{1}$ is
    \begin{equation}\label{BI11}
        \frac{|\mathbf{B}_{1}|^{(m-n+1)\beta/2-1}}{\mathcal{B}_{n}^{\beta}[a\beta /2, \kappa_{1};m\beta /2,\tau_{1}]}
        |\mathbf{I}_{n}-\mathbf{B}_{1}|^{(a-n+1)\beta/2-1}q_{\kappa_{1}}(\mathbf{I}_{n}-\mathbf{B}_{1})
        q_{\tau_{1}}(\mathbf{B}_{1})(d\mathbf{B}_{1}),
    \end{equation}
    where $\mathbf{I}_{n}-\mathbf{B}_{1} \in \mathfrak{P}_{n}^{\beta}$, also, we say that
    $\mathbf{B}_{1}$ has a \emph{matricvariate c-beta-Riesz type I distribution}.
  \item Similarly, assuming that $\mathbf{R} \sim \mathcal{P_{II}R}_{n \times m}^{\beta,II}(a,\kappa_{1}, \tau_{1},
    \boldsymbol{0}, \textbf{I}_{n},\textbf{I}_{m})$. Then the density of  $\mathbf{B}_{1}$ is
    \begin{equation}\label{BII11}
        \frac{|\mathbf{B}_{1}|^{(m-n+1)\beta/2-1}}{\mathcal{B}_{n}^{\beta}[a\beta /2,
        -\kappa_{1};m\beta /2,-\tau_{1}]}|\mathbf{I}_{n}-\mathbf{B}_{1}|^{(a-n+1)\beta/2-1}
        q_{\kappa_{1}}[(\mathbf{I}_{n}-\mathbf{B}_{1})^{-1}]
    \end{equation}
\par\noindent\hfill\mbox{$\times \ q_{\tau_{1}}[(\mathbf{B}_{1})^{-1}](d\mathbf{B}_{1}),$}\par\noindent %
    where $\mathbf{I}_{n}-\mathbf{B}_{1} \in \mathfrak{P}_{n}^{\beta}$. We say that
    $\mathbf{B}_{1}$ has a \emph{matricvariate k-beta-Riesz type I distribution}.
\end{enumerate}
\end{theorem}
\begin{proof} The proof is the same as that given in Theorem \ref{teo2}. \qed

Alternatively, observe that densities (\ref{BI11}) and (\ref{BII11}) can be obtained from densities
(\ref{BI1}) and (\ref{BII1}), respectively, making the following substitutions
\begin{equation}\label{ss}
    \mathbf{B} \rightarrow \mathbf{B}_{1}, \quad m \rightarrow n, \quad n \rightarrow m, \quad \nu \rightarrow
    a,
\end{equation}
and consequently $\kappa \rightarrow \kappa_{1}, \quad \tau \rightarrow \tau_{1},$ and $k_{i} \rightarrow
k_{1i} \quad t_{i} \rightarrow t_{1i}$.
\end{proof}

To end this section, below are obtained the non-standardised densities of the c-, and k-beta
distributions.

\begin{corollary}
Define  $\mathbf{C} = \u(\mathbf{\Theta})^{*}\mathbf{B}\u(\mathbf{\Theta})$, where $\u(\mathbf{\Theta})
\in \mathfrak{T}_{U}^{\beta}(m)$ is such that  $\mathbf{\Theta} =
\u(\mathbf{\Theta})^{*}\u(\mathbf{\Theta})$ is the Cholesky decomposition of $\mathbf{\Theta}$.
\begin{enumerate}
  \item Assuming that $\mathbf{B}$ has the density (\ref{BI1}), then the density of random matrix $\mathbf{C}$ is
    \begin{equation}\label{BGI}
        \propto|\mathbf{C}|^{(n-m+1)\beta/2-1} |\mathbf{\Theta}-\mathbf{C}|^{\nu-m+1)\beta/2-1}
        q_{\kappa}(\mathbf{\Theta}-\mathbf{C})q_{\tau}(\mathbf{C})(d\mathbf{C}),
    \end{equation}
    with constant of proportionally
    $$
      \frac{1}{\mathcal{B}_{m}^{\beta}[\nu\beta/2, \kappa;n\beta/2,\tau]
        |\mathbf{\Theta}|^{(\nu+n-m+1)\beta/2-1} q_{\kappa+\tau}(\mathbf{\Theta})},
    $$
    for $\mathbf{\Theta}-\mathbf{C}\in \mathfrak{P}_{m}^{\beta}$.
  \item Supposing that $\mathbf{B}$ has the density (\ref{BII1}), then the density of random matrix $\mathbf{C}$ is
    \begin{equation}\label{BGII}
        |\mathbf{C}|^{(n-m+1)\beta/2-1} |\mathbf{\Theta}-\mathbf{C}|^{\nu-m+1)\beta/2-1}
        q_{\kappa}[(\mathbf{\Theta}-\mathbf{C})^{-1}]q_{\tau}(\mathbf{C}^{-1})(d\mathbf{C}),
    \end{equation}
    with constant of proportionally
    $$
      \frac{1}{\mathcal{B}_{m}^{\beta}[\nu\beta/2, -\kappa;n\beta/2,-\tau]
        |\mathbf{\Theta}|^{(\nu+n-m+1)\beta/2-1} q_{\kappa+\tau}(\mathbf{\Theta}^{-1})},
    $$
    for $\mathbf{\Theta}-\mathbf{C}\in \mathfrak{P}_{m}^{\beta}$.
\end{enumerate}
\end{corollary}
\begin{proof}
This immediate from (\ref{hlt}).
\end{proof}

\section*{Conclusions}
Undoubtedly, in any generalisation of results there is a price to be paid, and in this case the price is
that of acquiring a basic understanding of some concepts of abstract algebra, which can initially be
summarised as the use of notation and a basic minimum set of definitions. However, we believe that a
detailed study of mathematical properties from a statistical standpoint can have a potential impact on
statistical theory. Furthermore, some statistical results in the literature have been studied. For
example \cite{mdm:06} address the problem of point estimation of parameters in complex shape theory.
Also, \cite{k:65} considered the estimation of parameters of a complex matrix multivariate normal
distribution and establishes a test of hypotheses about the mean. In a quaternionic context, \cite{bh:00}
set test statistics and their corresponding asymptotic distributions for two particular hypothesis tests.
As noted by the reviewer, the statistical results in \cite{m:82} and \cite{fz:90} can be extended, and in
fact are being extended to the case of real normed division algebras, but first they needed to study
various preliminary results, including those obtained in this work.

\section*{Acknowledgements}
This article was written under the existing research agreement between the first author and the
Universidad Aut\'onoma Agraria Antonio Narro, Saltillo, M\'exico. The second author was supported by a
joint research project among University of Medellin, University of Toulouse and University of Bordeaux,
France.

\end{document}